\documentclass{article}
\usepackage[utf8]{inputenc}
\usepackage[percent]{overpic}
\usepackage{amsmath}
\usepackage{amssymb}
\usepackage{tikz}
\usepackage{tikz-cd}
\usepackage{mathtools}
\usepackage{stmaryrd}
\usepackage{amsthm}
\usepackage{array} 
\usepackage[english]{babel}
\usepackage{enumitem}
\usepackage{lmodern}
\usepackage{xfrac}  

\usepackage[paper=a4paper, margin=2.5cm]{geometry}

\usepackage{hyperref}
\hypersetup{
    unicode=false,          
    pdftoolbar=true,        
    pdfmenubar=true,        
    pdffitwindow=false,     
    pdfstartview={FitH},    
    pdftitle={Gorenstein algebras and toric bundles},    
    pdfauthor={Khovanskii, A. G.; Monin, L.},     
    pdfkeywords= {Gorenstein algebras} {Toric bundles} {Spherical Varieties}  {Ring of conditions} 
    pdfnewwindow=true,      
    colorlinks=true,       
    linkcolor=blue,          
    citecolor=red,        
    filecolor=magenta,      
    urlcolor=cyan           
}

\theoremstyle{plain}
\newtheorem{theorem}{Theorem}[section]
\newtheorem{lemma}[theorem]{Lemma}
\newtheorem{proposition}[theorem]{Proposition}
\newtheorem{corollary}[theorem]{Corollary}

\theoremstyle{definition}
\newtheorem{definition}[theorem]{Definition}
\newtheorem{remark}[theorem]{Remark}
\newtheorem{example}[theorem]{Example}

\newcommand{\rleft}{\mathopen{}\mathclose\bgroup\left}
\newcommand{\rright}{\aftergroup\egroup\right}

\newcommand{\C}{{\mathbb{C}}}

\newcommand{\kk}{{\normalfont\mathbf{k}}}
\newcommand{\R}{{\mathbb{R}}}
\newcommand{\Z}{{\mathbb{Z}}}

\newcommand{\Q}{{\mathbb{Q}}}

\newcommand{\m}{{\mathfrak{m}}}

\newcommand{\Dm}{{\mathcal{D}}}

\newcommand{\Pm}{{\mathcal{P}}}

\newcommand{\cL}{{\mathcal{L}}}

\newcommand{\Ann}{{\mathrm{Ann}}}

\newcommand{\diff}{\mathop{}\!d}
\newcommand{\Diff}{{\mathrm{Diff}}}

\newcommand{\Hom}{{\mathrm{Hom}}}

\newcommand{\Sym}{{\mathrm{Sym}}}

\newcommand{\Exp}{{\mathrm{Exp}}}

\usepackage{pgf,tikz}
\usepackage{mathrsfs}
\usetikzlibrary{arrows}
\definecolor{uuuuuu}{rgb}{0.26666666666666666,0.26666666666666666,0.26666666666666666}
\definecolor{yqyqyq}{rgb}{0.5019607843137255,0.5019607843137255,0.5019607843137255}
\definecolor{uququq}{rgb}{0.25098039215686274,0.25098039215686274,0.25098039215686274}

\title{Gorenstein algebras and toric bundles}
\author{Askold Khovanskii}
\author{
  Askold Khovanskii
  \and
  Leonid Monin
}
\date{}

\newcommand{\Addresses}{{
  \bigskip
  \footnotesize

  A.~Khovanskii, \textsc{Department of Mathematics, University of Toronto, Toronto, Canada; Moscow Independent University, Moscow, Russia.}\par\nopagebreak
  \textit{E-mail address}: \texttt{askold@math.utoronto.ca}

  \medskip

  L.~Monin (Corresponding author), \textsc{Max Planck Institute for Mathematics in the Sciences, Leipzig, Germany}\par\nopagebreak
  \textit{E-mail address}: \texttt{leonid.monin@mis.mpg.de}
}}

\makeatletter
\newcommand{\subjclass}[2][1991]{%
  \let\@oldtitle\@title%
  \gdef\@title{\@oldtitle\footnotetext{#1 \emph{Mathematics subject classification.} #2}}%
}
\newcommand{\keywords}[1]{%
  \let\@@oldtitle\@title%
  \gdef\@title{\@@oldtitle\footnotetext{\emph{Key words and phrases.} #1.}}%
}
\makeatother


\begin{document}
\maketitle
\begin{abstract}
 We study commutative algebras with Gorenstein duality, i.e. algebras $A$ equipped with a non-degenerate bilinear pairing such that $\langle ac,b\rangle=\langle a,bc\rangle$ for any $a,b,c\in A$. If an algebra $A$ is Artinian, such pairing exists if and only if $A$ is Gorenstein. We give a description of algebras with Gorenstein duality as a quotients of the ring of differential operators by the annihilator of an explicit polynomial (or more generally formal polynomial series). This provides a calculation of Macaulay's inverse systems for (not necessarily Artinian) algebras with Gorenstein duality. Our description generalizes previously know description of graded algebras with Gorenstein duality generated in degree 1.

 Our main motivation comes from the study of even degree cohomology rings. In particular, we apply our main result to compute the ring of cohomology classes of even degree of toric bundles and the ring of conditions of horospherical homogeneous spaces. 
\end{abstract}


\section{Introduction}
\label{sec:intro}

In this paper we study Macaulay's inverse systems of algebras with Gorenstein duality. We say that an algebra $A$ over a field $\kk$ has a \emph{Gorenstein duality} if there exists a non-degenerate bilinear pairing on $A$ such that  $\langle ac,b\rangle=\langle a,bc\rangle$ for any $a,b,c\in A$. If algebra $A$ is a finite dimensional vector space over $\kk$, such a pairing exists if and only if $A$ is a Gorenstein ring.

Let $\kk$ be a field of characteristic $0$ and consider a polynomial ring $R=\kk[x_1,\ldots,x_n]$. Recall that the \emph{polynomial differential operator with constant coefficients} is a differential operator on $R$ of the form
\[
D= \sum_{i_1,\ldots i_n} a_{i_1}\ldots a_{i_n} \left(\partial/\partial x_1\right)^{i_1}\ldots\left(\partial/\partial x_n\right)^{i_n}.
\]
The ring of differential operators with constant coefficients forms a polynomial ring $\kk[\partial/\partial x_1,\ldots,\partial/\partial x_n]$ which acts naturally on $R$.

For an ideal $I\subset\kk[\partial/\partial x_1,\ldots,\partial/\partial x_n]$ its \emph{inverse system} is an ideal $I^\perp\subset R$ of polynomials which annihilate~$I$:
\[
I^\perp = \big\{f \,\mid\, D\cdot f = 0, \forall D\in I\big\}\subset R.
\]
Inverse systems were introduced by Macaulay (see \cite{Mac1,Mac2, Gr}) and is a convenient tool both in commutative algebra and analysis (see \cite{m1,m2,m3,b1,b2} for some recent developments). In particular, the quotient algebra $\kk[x_1,\ldots,x_n]/I$ is a zero-dimensional Gorenstein ring if and only if $I^\perp$ is a principle ideal \cite[Exercise 21.7]{Eis}.

For a finite dimensional vector space $V$ over $\kk$, one can identify the ring of differential operators with constant coefficients on $V$ with symmetric algebra $\Sym(V)$. For a polynomial function $f: V \to \kk$, let us denote by $\Ann(f):= \langle f \rangle^\perp$ its annihilator in the ring of differential operators with constant coefficients. Thus for any polynomial $f$ on $V$, the factor algebra $\Sym(V)/\Ann(f)$ has a Gorenstein duality, and every zero-dimensional Gorenstein factor algebra of $\Sym(V)$ is constructed in this way.

In this paper we extend these results in the following way. First, we extend the above correspondence to non-Artinian commutative algebras with Gorenstein duality over a field of characteristic $0$. We show that for a (possibly infinite dimensional) vector space $V$ the factor algebras of $\Sym(V)$ with Gorenstein duality are in bijection with \emph{formal polynomial series} on $V$ (Corollary~\ref{cor:gorfact}). We also give an explicit formula for the generator of the  Macaulay's inverse system of factor algebras of $\Sym(V)$ with Gorenstein duality (Theorem~\ref{thm:potential}). 

Our main motivation to study algebras with Gorenstein duality comes from computation of cohomology rings. Indeed, let $M$ be a smooth closed orientable manifold of even dimension and let 
\[
A^*(M) := \bigoplus_{i=0}^{(\dim_\R M)/2} H^{2i}(M,\R),
\]
be the ring of cohomology classes of even degree. Then $A^*(M)$ is a commutative ring with Gorenstein duality given by Poincar\'e duality on $H^*(M)$. Thus our description of algebras with Gorenstein duality provides a way to compute $A^*(M)$. 

Such approach was used to compute the cohomology ring of smooth projective toric varieties in \cite{KP} and the cohomology rings of full flag varieties in \cite{KavehVolume}. These works, however, crucially use that the above cohomology rings are generated in degree 2. Our explicit calculation of Macaulay's  inverse systems of algebras with Gorenstein duality (Theorem~\ref{thm:potential}) is applicable for computation of much wider class of cohomology rings. As an illustration, we apply Theorem~\ref{thm:potential} to computation of the rings of cohomology classes of even degrees of toric bundles (Theorem~\ref{thm:cohtorbun}). In particular, this includes the computation of cohomology rings of smooth, toroidal horospherical varieties and the ring of conditions of horospherical homogeneous spaces (Theorem~\ref{thm:cond}).

\subsection{Organization of the paper}
Section~\ref{sec:gor} provides the necessary background on algebras with Gorenstein duality. In Section~\ref{sec:poly} we define spaces of polynomials and formal polynomial series on (possibly infinite dimensional) vector space $V$ and study the action of differential operators with constant coefficients on these spaces.
In Section~\ref{sec:inverse} we prove our main result on the explicit calculation of Macaulay's  inverse systems of algebras with Gorenstein duality. We also consider two important examples of local algebras and graded algebras with duality. Finally,
Section~\ref{sec:toric} is devoted to the application of our main results to the computation of cohomology rings of toric bundles.

\subsection*{Acknowledgements}
We thank Ben Briggs and \"Ozg\"ur Esentepe for several helpful conversations.
The first author is partially supported by the Canadian Grant No.~156833-17.
Big part of this project was done while the second author was at a research stay at MPI MiS, he thanks MPI MiS for its hospitality.
 

\section{Algebras with Gorenstein duality}\label{sec:gor}
In this section we study (not-necessarily Artinian) algebras with Gorenstein duality and describe a construction of such algebras recently obtained in \cite{uspehi}.

\subsection{Algebras with Gorenstein duality} Let $\kk$ be a field of characteristic $0$, in this paper we assume $\kk= \Q, \R$ or $\C$. Let $A$ be a commutative algebra with identity over $\kk$. A linear function $\ell \colon A\to \kk$ defines a symmetric, bilinear pairing on $A$ via:
\[
\langle a,b\rangle_{\ell} := \ell(a\cdot b) \text{ for any } a, b\in A. 
\]
\begin{definition}
  A pairing $\langle\cdot,\cdot\rangle_\ell$ on algebra $A$ is called \emph{Gorenstein duality} if it is non-degenerate. 
\end{definition}

The following lemma gives another characterization of Gorenstein duality.
\begin{lemma}
  Let $\ell\colon  A\to \kk$ be any linear function, then
  \[
  \langle ab,c\rangle_\ell =\langle a,bc\rangle_\ell \quad \forall a,b,c\in A.\footnote{For a non-commutative algebras $A$ such non-degenerate pairing is usually called Frobenius pairing. However, since we work with commutative algebras, we use term Gorenstein duality.}
  \]
  Moreover, any symmetric bilinear pairing $\langle\cdot,\cdot\rangle$ such that $\langle ab,c\rangle=\langle a,bc\rangle$ for any $a,b,c\in A$ is given by some linear function.
\end{lemma}
\begin{proof}
For the first part it is enough to notice that for all  $a,b,c\in A$
\[
\langle ab,c\rangle_\ell = \ell(abc) = \langle a,bc\rangle_\ell.
\]
For the second part, define function $\ell\colon A\to \kk$ via $\ell(a) = \langle 1,a\rangle$.
Since $\langle\cdot,\cdot\rangle$ is symmetric and bilinear, the function $\ell$ is linear. Moreover, for any $a,b\in A$
\[
\langle a,b \rangle = \langle 1,ab \rangle = \ell(ab) = \langle a,b \rangle_\ell,
\]
which finishes the proof.
\end{proof}

In the case when algebra $A$ is Artinian, the existence of Gorenstein duality implies that $A$ is a Gorenstein algebra (in fact these two properties are equivalent). In what follows we will not assume an algebra $A$ to be Artinian, in this case an algebra with a Gorenstein duality is not necessarily Gorenstein.

Let $W \subset A$ be a (possibly infinite dimensional) vector space which generates algebra $A$ multiplicatively, i.e. every element $a\in A$ can be represented as a finite sum:
\[
a= P_1(v_1,\ldots,v_k),
\]
 with $v_i\in V$ and $P\in \kk[x_1,\ldots, x_k]$ is a polynomial with coefficients in $\kk$. Let further $\pi\colon V\to W$ be a surjective linear map. A map $\pi$ induces an algebra homomorphism
\[
\pi\colon\Sym(V)\to A, 
\]
which we will denote with the same symbol. Since $W$ generates algebra $A$, the homomorphism $\pi$ above is surjective. Let $\widetilde \ell = \ell\circ \pi$ be a linear function on $\Sym(V)$ induced by $\ell$.  Then the following description of commutative algebras with Gorenstein duality was obtained in \cite{uspehi}.

\begin{theorem}[\cite{uspehi}]\label{thm:gor}
 The kernel of the map $\pi\colon\Sym(V)\to A$ is given by
 \[
 \ker(\pi) = \{P\in \Sym(V) \,|\, \widetilde\ell(P\cdot Q)=0 \text{ for any } Q\in \Sym(V)\},
 \]
 where $\widetilde\ell= \pi^* \ell$. Moreover, for  a (possibly infinite dimensional) vector space $V$, factor algebras of $\Sym(V)$ with chosen Gorenstein duality are in bijection with linear functions on $\Sym(V)$. 
\end{theorem}

For a linear function $\widetilde\ell$, on $\Sym(V)$, we will denote by  $I_{\widetilde \ell}$ the corresponding ideal:
 \[
I_{\widetilde \ell} =\{P\in \Sym(V) \,|\, \widetilde\ell(P\cdot Q)=0 \text{ for any } Q\in \Sym(V)\}
 \]
Theorem~\ref{thm:gor} implies that the classification of algebras with Gorenstein duality boils down to a description of the dual space of $\Sym(V)$.

\section{Polynomials and formal polynomial series on infinite dimensional vector spaces}\label{sec:poly}
 In this section we give definitions of polynomial maps and formal polynomial series on an infinite dimensional vector space. We also define an action of symmetric algebra $\Sym(V)$ of a vector space $V$ on the space of polynomial maps and formal polynomial series on $V$. 

\subsection{Polynomial maps between infinite dimensional vector spaces} 

The notion of polynomial function on infinite dimensional vector space was used by Pukhlikov and the first author, our approach closely follows the one in \cite{PK92}. 

\begin{definition}
  We will call a function $f:V\to \kk$ a homogeneous polynomial of degree $d$ if its restriction to any finite dimensional vector subspace  of $V$ is a homogeneous polynomial of degree $d$. A polynomial on $V$ is a finite sum of homogeneous polynomials.
\end{definition}

More generally, we will need a notion of a polynomial map between two (not-necessarily finite dimensional) vector spaces.

\begin{definition}\label{def:polycomp}
A map $f:V\to W$ between two vector spaces is called a polynomial map of degree $d$, if for any linear function $\ell \in W^\vee$ the composition $\ell\circ f:V\to \kk$ is a polynomial of degree $d$.
\end{definition}

One has an equivalent definition of a polynomial map between two vector spaces.

\begin{definition}\label{def:polydirect}
  A map $f:V\to W$ between two vector spaces is called a polynomial map of degree $d$, if for every finite dimensional subspace $V'\subset V$ the image $f(V')$ is contained in a finite dimensional subspace $W'\subset W$ and the restriction $f:V'\to W'$ is a polynomial map of degree $d$ (in the usual sense).
\end{definition}

\begin{lemma}\label{lem:polyeq}
Definitions~\ref{def:polycomp} and~\ref{def:polydirect} of polynomial maps between two vector spaces are equivalent.
\end{lemma}
\begin{proof}
  Assume $f:V\to W$ is a polynomial map in the sense of Definitions~\ref{def:polydirect} and let $V'\subset V$ be a finite dimensional vector space. Then the map $f:V'\to span(f(V'))$ is a polynomial map between two finite dimensional vector spaces and, therefore, $\phi\circ f$ is a polynomial on $V'$ for any $\phi\in W^\vee$. 
  
  For the other direction assume $f:V\to W$ is a polynomial map in the sense of Definitions~\ref{def:polycomp} and let $V'\subset V$ be a finite dimensional vector space.  First we will show that $W' = span(f(V'))$ is finite. Indeed, since for any linear function $\phi\in W^\vee$, the compositions $\phi\circ f$ is a polynomial on $V'$ of degree less or equal to $d$, we have a linear map:
  \[
  (W')^\vee \to \Sym_{\leq d}\left((V')^\vee\right).
  \]
  Moreover, since $W' = span(f(V'))$, the map above is injective, and hence $\dim W'\leq \dim \Sym_{\leq d}\left((V')^\vee\right)< \infty$. The restriction $f:V'\to W'$ is polynomial since its composition with coordinate functions (with respect to any basis of $W'$) are polynomials. 
\end{proof}

The following construction of the polynomial mappings will be very important for this paper. Let $V$ be a vector space over $\kk$ and let $P = a_0+a_1x+\ldots+a_nx^n\in \kk[x]$ be a univariate polynomial. Then there is a polynomial mapping $\widetilde{P}:V\to \Sym(V)$ defined via:
\[
\widetilde{P}(v) = a_0+a_1v+\ldots+a_nv^n \in \Sym(V)
\]
The Example~\ref{ex:unipoly} below is a slight generalisation of this construction.

\begin{example}\label{ex:unipoly}
  Let $A$ be an algebra with identity and $\phi\colon V\to A$ be a linear map. Then every univariate polynomial $P= a_0+ a_1x +\ldots + a_n x^n\in A[x]$ with coefficients in $A$ defines a polynomial mapping $V\to A$ via
  \[
  \widetilde{P} (v) = a_0 + a_1\phi(v) + a_2 \phi(v)^2 +\ldots + a_n \phi(v)^n \in A. 
  \]
  In particular, every polynomial $P\in \kk[x]$ with coefficients in $\kk$ defines a polynomial map $\widetilde{P}:V\to A$.
\end{example}

\subsection{Action of symmetric algebra on the space of polynomial maps}

For a (possibly infinite dimensional) vector space $V$ we denote by $\Sym(V)$ its symmetric algebra, i.e. the quotient of tensor algebra of $V$ by two-sided ideal generated by elements of the form $v\otimes w-w\otimes v$. We identify $\Sym(V)$ with the space $\Diff(V)$ of differential operators with constant coefficients on $V$ in the following way. The zero degree component $\Sym_0(V)\simeq \kk$ corresponds to operators of multiplication by a number. For elements of positive degree we define the map
\[
\mathcal{D}\colon v_1\ldots v_k \mapsto L_{v_1}\ldots L_{v_k},
\]
where $L_v$ is the Gateaux derivative in the direction $v\in V$. For a map $f:V\to W$ its Gateaux derivative in the direction $v\in V$ is given by
\begin{equation}\label{eq:Gatoux}
    L_v f (x) = \lim_{t\to 0} \frac{f(x+tv)-f(x)}{t},
\end{equation}
if the limit exists (\cite{gateaux}).

\begin{remark}
For a general map $f:V\to W$ for the limit on the right had side of (\ref{eq:Gatoux}) to be defined, one need to choose a topology on the target vector space $W$. In this paper we would work with polynomial mappings, so for fixed $v$ and $x$, the image $f(x+tv)$ always belong to a finite dimensional subspace of $W$ and hence the limit $\lim_{t\to 0} \frac{f(x+tv)-f(x)}{t}$ is well defined. In other words, we endow $W$ with the topology for which a subset $U\subset W$ is open if and only if $U\cap W'$ is open for every finite dimensional vector subspace $W'\subset W$.
\end{remark}

\begin{remark}
Since for the computation of the value of the Gateaux derivative $L_v f(x)$ of a polynomial mapping at a fixed point $x\in V$ we need to work with only finite dimensional vector spaces, we can define the  Gateaux derivative purely formally. This allows to extend all our results to any field $\kk$ of characteristic $0$.
\end{remark}

For an element $P\in \Sym(V)$, we denote by $\Dm_P$ the corresponding differential operator. It is easy to check that for a polynomial map $f:V\to W$ of degree $d$ and a vector $v\in V$, the Gateaux derivative $L_v f$ is a polynomial map of degree $d-1$ or 0. Therefore, the space of polynomial mappings has a structure of a graded module over $\Diff(V)$.

\begin{example}\label{ex:polar}
  Let $f:V\to \kk$ be a homogeneous polynomial of degree $d$, then for any $v\in V$, we have $\Dm_{v^d}\cdot f:= (L_v)^d \cdot f \equiv d! f(v)$ is a constant function.
\end{example}

\begin{example}[Continuation of Example~\ref{ex:unipoly}]\label{ex:Gatoux}
    Let $A$ be an algebra with identity and $\phi\colon V\to A$ be a linear map. Let $\widetilde{P}$ be a polynomial map $\widetilde P:V\to A$ given by a univariate polynomial $P\in A[x]$. Let further $v\in V$ be a vector, then one has 
    \[
    L_v \widetilde P = \phi(v)\cdot \widetilde{\partial_x P},
    \]
    where $ \widetilde{\partial_x P}$ is a polynomial map given by the usual derivative $\partial_x P$ of $P$. By linearity, it is enough to check the formula above for $P=a_ix^i$ being a monomial in $A[x]$. So we have:
    \[
    L_v\cdot (\widetilde{a_ix^i}) = \lim_{t\to 0} \frac{a_i\cdot\phi(x+tv)^i-a_i\cdot\phi(x)^i}{t} = \lim_{t\to 0}a_i\cdot\frac{(\phi(x)+t\phi(v))^i-\phi(x)^i}{t} = a_i\cdot \phi(v)\cdot i \phi(x)^{i-1}.
    \]
\end{example}

\subsection{Formal polynomial series}
In this subsection we introduce a more general notion of \emph{formal polynomial series} and \emph{formal polynomial series with values in a vector space $W$}. 

\begin{definition}
  For (not-necessarily finite dimensional) vector spaces $V$ and $W$, a formal polynomial series on $V$ with values in $W$ is a formal sum
\[
f = f_0 + f_1 + \ldots + f_i +\ldots, 
\]
where $f_i:V\to W$ is a homogeneous polynomial map of degree $i$. A formal polynomial series with values in $\kk$ is just called a \emph{formal polynomial series.}
\end{definition}

We can extend the action of $\Diff(V)$ on the space of formal 
polynomial series  $\kk\llbracket V \rrbracket$ via term-wise differentiation:
\[
 \Dm_P \left(\sum_{i=0}^\infty f_i\right) = \sum_{i=r}^\infty \Dm_P(f_i), 
\]
for a homogeneous polynomial $P\in \Sym_r(V)$  extended by linearity for general $P\in \Sym(V)$. For a formal polynomial series $F=\sum_{i=0}^\infty f_i$, let us denote by $F(0)$ its constant term $f_0\in \kk$.

As in the case of polynomial maps, in this paper we are mostly interested in the following special class of formal polynomial series.

\begin{example}[Continuation of Example~\ref{ex:unipoly}]\label{ex:power}
  As before, let $A$ be an algebra with identity and $\phi\colon V\to A$ be a linear map. Then every univariate formal power series $f= \sum_{i=0}^\infty a_i x^i\in A\llbracket x\rrbracket$ with coefficients in $A$ defines a formal polynomial series on $V$ with values in $A$ via
  \[
  \widetilde{f} = a_0 + a_1\phi(v) + a_2 \phi(v)^2 +\ldots + a_n \phi(v)^n + \ldots. 
  \]
  Note that each summand in the expression above is indeed a homogeneous polynomial mapping from $V$ to $A$, so their (formal) sum is a formal polynomial series with values in $A$.
 
 In particular, every formal power series every polynomial $f\in \kk\llbracket x \rrbracket$ with coefficients in $\kk$ defines a polynomial map $\widetilde{P}:V\to A$.
\end{example}

As before, the action of $\Sym(V)$ on formal polynomial series as in Example~\ref{ex:power} has an easy description.

\begin{proposition}\label{prop:der}
  Let $A$ be an algebra and $\phi\colon V\to A$ be a $\kk$-linear map. Let $\widetilde{f}$ be a formal polynomial series with values in $A$  given by a univariate power series $f(x)= \sum_{i=0}^\infty a_ix^i$ with $a_i\in A$. 
  Then for any $v\in V$, we have
  \[
  L_v  \widetilde{f} = \phi(v) \cdot \widetilde{\partial_x f},
  \]
  where $\widetilde{\partial_x f}$ is the formal polynomial series with values in $A$ given by the derivative $\partial_x f$ of $f$.
\end{proposition}

\begin{proof}
Since we apply the differentiation to $\widetilde f$ term-wise, it is enough to show that $L_v (a_i\phi(x)^i) = \phi(v)\cdot ia_i\phi(x)^{i-1}$, which was done in Example~\ref{ex:Gatoux}.
\end{proof}

As an immediate corollary, we obtain the following statement, which is the central part of this section.

\begin{corollary}
Let $A$ be an algebra and $\phi\colon V\to A$ be a $\kk$-linear map. Let $\Exp(x)$ be a formal polynomial series with values in $A$ given by a exponential power series $\sum_{i=0}^\infty \frac{x^i}{i!}$. Then for any $P\in \Sym(V)$
  \[
  \Dm_P \cdot \Exp(x) =  \phi(P) \cdot \Exp(x).
  \]
\end{corollary}
\begin{proof}
 The statement follows from successive application of Proposition~\ref{prop:der} and the fact that the derivative of the exponential power series is itself.
\end{proof}

Let us denote by $\kk\llbracket V \rrbracket$ the space of formal polynomial series on $V$. Given a formal polynomial series $f$ on $V$ with values in $W$, one can pullback the linear functions on $W$ to formal polynomial series on $V$. In other words one has a map:
\[
f^*: W^\vee \to \kk\llbracket V \rrbracket, \quad \ell \mapsto \ell \circ f_0+\ell\circ f_1 +\ldots 
\]
By Definition~\ref{def:polycomp} each of the compositions $\ell\circ f_i$ is a homogeneous polynomial of degree $i$, so the sum on the right hand side is indeed a formal polynomial series.

 We finish this section with the proof that the space of formal polynomial series is isomorphic to the dual vector space to symmetric algebra $\Sym(V)^\vee$. First we will need the following definition.

\begin{definition}\label{def:pair}
 Define a bilinear pairing $B$ between $\Diff(V)$ ($\simeq \Sym (V)$) and $\kk\llbracket V \rrbracket$ in the following way
  \[
  \langle \Dm_P, F\rangle := (\Dm_P\cdot F)(0),
  \]
  where $(\Dm_P\cdot F)(0)$ as before denotes the constant term of a formal polynomial series.
\end{definition}
The pairing from Definition~\ref{def:pair} defines a map $B: \kk\llbracket V \rrbracket \to \Sym(V)^\vee$ via $f \mapsto \langle\cdot,f\rangle.$

\begin{theorem}\label{prop:dual}
Let $V$ be a vector space and let as before, $\Exp$ be a formal polynomial map with values in $\Sym(V)$ given by the exponential power series. Then the map 
\[
\Exp^*: \Sym(V)^\vee \to \kk\llbracket V \rrbracket
\]
is the inverse of the map $B: \kk\llbracket V \rrbracket \to \Sym(V)^\vee$ induced by the pairing in Definition~\ref{def:pair}. In particular, both maps are isomorphisms of vector spaces.

\end{theorem}
\begin{proof}
Since both maps are linear, it is enough to check that $\Exp^*\circ B\colon \kk\llbracket V \rrbracket \to \kk\llbracket V \rrbracket$ is identity. This follows from a direct computation and Example~\ref{ex:polar}. Indeed let $f=\sum_{i=0}^\infty f_i$ be a formal polynomial series, then
\[
\left(\Exp^*\circ B\right)(f) = \sum_{i=0}^\infty \Dm_{\frac{x^i}{i!}}\cdot f(0)= \sum_{i=0}^\infty \Dm_{\frac{x^i}{i!}}\cdot f_i = \sum_{i=0}^\infty f_i(x) =f.
\]

\end{proof} 

\begin{corollary}\label{cor:gorfact}
For  a (possibly infinite dimensional) vector space $V$, factor algebras of $\Sym(V)$ with chosen Gorenstein duality are in bijection with formal polynomial series on $V$. 
\end{corollary}
\begin{proof}
Statement follows from Theorem~\ref{thm:gor} and the fact that $\Sym(V)^\vee \simeq \kk\llbracket V \rrbracket$.
\end{proof}

Another immediate corollary of Theorem~\ref{prop:dual} is the following statement.
\begin{corollary}\label{cor:exp}
For any linear function $\ell\in \Sym(V)^\vee$ we have
\[
\ell (P) = \langle\Dm_P, \Exp^*(\ell) \rangle,
\]
where $\Exp$ is the exponential formal polynomial series with values in $\Sym(V)$.
\end{corollary}

\section{Explicit description of Macaulay's inverse system}\label{sec:inverse}
In this subsection we obtain an explicit description of Macaulay's inverse system for algebras with Gorenstein duality over field of characteristic $0$. First we need the following lemma.

\begin{lemma}\label{lem:potential}
  Let $\widetilde \ell\in \Sym(V)^\vee$ be a linear function on $\Sym(V)$. Then the ideal $I_{\widetilde \ell}$ coincides with 
  \[
  \Ann(\Exp^*{\widetilde\ell}) = \{P\in \Sym(V)\,|\, \Dm_P\cdot \Exp^*{\widetilde\ell}=0 \}.
  \]
\end{lemma}

\begin{proof}
If $P\in\Ann(\Exp^*{\widetilde\ell})$, then $\widetilde \ell (Q\cdot P)= \langle Q\cdot P,\Exp^*{\widetilde\ell}\rangle =0$, for any $Q\in \Sym(V)$. Hence $P\in I_{\widetilde\ell}$. 

For the other direction, notice that if $\Dm_P \cdot \Exp^*{\widetilde\ell} \ne 0$, then there exists at least one non-zero component of  $\Dm_P \cdot \Exp^*{\widetilde\ell}$, and therefore there exist $Q$ such that the constant term of $\Dm_{Q\cdot P}\cdot\Exp^*{\widetilde\ell}=\Dm_Q\cdot \Dm_p \cdot\Exp^*{\widetilde\ell}$ is non zero. Thus if $P\notin \Ann(\Exp^*{\widetilde\ell})$, we also have $P\notin I_{\widetilde \ell}$. \qedhere
\end{proof}

Now we are ready to give a construction of algebras with Gorenstein duality. As before, let $A$ be a commutative algebra, $\ell$ be a linear function on $A$ which induces a non-degenerate pairing $\langle\cdot,\cdot\rangle_\ell$. Let further $W\subset A$ be a vector subspace which generates $A$ multiplicatively and let $\pi:V\to W$ be a surjective linear mapping. The immediate corollary of Theorem~\ref{thm:gor} and Lemma~\ref{lem:potential} is the following theorem.

\begin{theorem}\label{thm:potential}
Let $A$ be a commutative algebra over $\kk$ with Gorenstein duality given by linear function $\ell:A\to \kk$. Let $V$ be a (possibly infinite dimensional) vector space and let $\pi:\Sym(V) \to A$ be surjection. Then we have
 \[
 A\simeq \Sym(V)/\Ann(\Exp^*{\widetilde\ell}),
 \]
 where $\Ann(\Exp^*{\widetilde\ell}) = \{P\in \Sym(V)\,|\, \Dm_P\cdot \Exp^*{\widetilde\ell}=0 \}$.
\end{theorem}

For an algebra $A$ with duality given by linear function $\ell$ and a surjection $\pi:\Sym(V)\to A$, we will call the formal polynomial series $\Exp^*{\widetilde \ell}$ \emph{the potential} of $A$ with respect to the generating space $V$.

\subsection{Graded algebras with Gorenstein duality} In this subsection we consider important partial cases of algebras with Gorenstein duality whose potentials are actual polynomials. 

Let $A$ be a local algebra with maximal ideal $\m$ and the residue field $A/\m\simeq \kk$ with duality defined by a function $\ell$. Assume further that $\m^d = 0$ for some $d$. Let $W\subset \m$ be a subspace which generates $A$ multiplicatively and let $V\to W$ be a surjective map.

\begin{proposition}
 For $A$ and $V$ as above, the potential $\Exp^*{\widetilde\ell}$ of $A$ on $V$ has only finitely many non-zero sumands. In over words, $\Exp^*{\widetilde\ell}$ is a polynomial on $V$.
\end{proposition}
\begin{proof}
  Proposition immediately  follows from the fact that $\pi(x^i)=0$ for any $x\in V$ and $i\geq d$.
\end{proof}

The most important example of local algebras with Gorenstein duality for us are graded algebras. Let
\[
A= \bigoplus_{i=0}^n A_i,
\]
be a commutative graded algebra over $\kk$, such that 
\begin{itemize}
    \item $A_0 \simeq A_n \simeq \kk$;
    \item the natural bilinear pairing $A_i\times A_{n-i} \to A_n\simeq k$ is non-degenerate for any $i = 0, \ldots, n$.
\end{itemize}
Let us fix an isomorphism $\ell:A_n\simeq \kk$ and extend it to the whole $A$ trivially, i.e. $\ell(A_i)=0$, for $i\ne n$. Then a function $\ell$ provides a Gorenstein duality.

Let $W \subset A$ be a (possibly infinite dimensional) vector space which contains multiplicative generators of $A$. Then  $W$ has a natural structure of a graded vector space:
\[
W = \bigoplus_{i=0}^n W_i, \quad W_i = W\cap A_i.
\]
Let further $\pi:V\to \bigoplus_{i=1}^n W_i$ be a surjectiction to the subspace of $W$ of elements of positive degree. One can pullback the grading of $W$ to a grading of $V$:
\[
V = \bigoplus_{i=1}^n V_i, \quad V_i=\pi^{-1}(W_i).
\]

As before, the surjection $\pi:V\to \bigoplus_{i=1}^n W_i$ extends to a surjective homomorphism of algebras $\pi:\Sym(V) \to A$. Moreover, the grading on $V$ defines a grading on $\Sym(V)$ via
\begin{equation}\label{eq:grading}
    \deg \left(v_1^{k_1}\ldots v_n^{k_n}\right) = \sum_{i=1}^n ik_i,
\end{equation}
for any $v_i\in V_i$. This grading on $\Sym(V)$ makes the homomorphism $\pi\colon \Sym(V) \to A$ graded. Let 
\[
\Sym(V) = \bigoplus_{i=0}^\infty \Sym_i(V),
\]
be a decomposition into quasi-homogeneous components (i.e. homogeneous with respect to the grading from (\ref{eq:grading})). We will say that a linear function $\widetilde \ell:\Sym(V)\to \kk$ is $n$ quasi-homogeneous if
\[
\widetilde \ell(\Sym_i(V)) = 0, \text{ for } i\ne n.
\]

\begin{theorem}\label{thm:grgor}
 Let $V$ be a (possibly infinite dimensional) graded vector space. Then graded Gorenstein factor algebras of $\Sym(V)$ are in bijection with quasi-homogeneous linear functions on $\Sym (V)$ up to scaling.
\end{theorem}
\begin{proof}
  Let $\pi:\Sym(V)\to A$ be a surjection to a graded algebra $A=\bigoplus_{i=0}^n A_i$ with duality and let $\ell:A_n\to \kk$, be some isomorphism (which is unique up to scaling). Then $\widetilde \ell = \ell\circ\pi$ is an $n$ quasi-homogeneous linear function, and by Theorem~\ref{thm:gor}, we have
  \[
   A \simeq  \Sym(V)/I_{\widetilde \ell}.
  \]
  
  For the other direction let $\widetilde \ell$ be an $n$ quasi-homogeneous linear function on $\Sym(V)$. We will show that the ideal $I_{\widetilde\ell}$ is quasi-homogeneous and hence the factor algebra $\Sym(V)/I_{\widetilde\ell}$ is a graded Gorenstein algebra.
  
  Let $P\in I_{\widetilde \ell}$, i.e. $\widetilde \ell(P\cdot Q)$ for any $Q\in \Sym(V)$. Let  $P=\sum_{i=0}^k P_i$ be a decomposition of $P$ into quasi homogeneous components. Then for any $Q\in \Sym(V)$, we have
  \[
    \widetilde\ell(P_i\cdot Q) = \sum_j \widetilde\ell(P_i\cdot Q_j) = \widetilde\ell(P_i\cdot Q_{n-i}),
  \]
  where $Q=\sum_j Q_j$ is the decomposition into quasi-homogeneous components. But since $P\in I_{\widetilde\ell}$ and $\widetilde \ell$ is $n$ quasi-homogeneous we have
  \[
   0= \widetilde\ell(P\cdot Q_{n-i}) = \widetilde\ell(P_i\cdot Q_{n-i}).
  \]
  Thus $\widetilde\ell(P_i\cdot Q)=0$ for any $Q\in \Sym(V)$ and the ideal $I_{\widetilde\ell}$ is quasi-homogeneous.
\end{proof}

Now we will give a second version of Theorem~\ref{thm:grgor} in terms of potentials. First we will need the following definition.

\begin{definition}
 We will call a function $f\colon V\to \kk$ on a graded vector space a quasi-homogeneous polynomial of degree $d$ if its restriction to any finite dimensional vector subspace $W$ of $V$ which decomposes into graded components is a quasi-homogeneous polynomial of degree $d$.
\end{definition}

In case when $V=V_1$ is a graded vector space with trivial graded structure, quasi homogeneous polynomials on $V$ are just homogeneous polynomials. In general we would need the following technical lemma.

\begin{lemma}\label{lem:quasiispoly}
Let $V$ be a graded vector space, and let $f\colon V\to \kk$ be quasi-homogeneous polynomial. Then $f$ is a polynomial.
\end{lemma}
\begin{proof}
  It is enough to show that for every finite dimensional $W\subset V$ the restriction of $f$ to $W$ is a polynomial. But for every such $W$ one can construct a finite dimensional decomposable $\widetilde W \subset V$ with $W\subset \widetilde W$.  Since the restriction of $f$ to $\widetilde W$ is a polynomial, such is its restriction to $W$.
\end{proof}

\begin{lemma}\label{lem:qhpoly}
  Let $\widetilde\ell$ be an $n$ quasi-homogeneous linear function. Then the formal polynomial series $\Exp^*{\widetilde\ell}$ is a quasi-homogeneous polynomial of degree $n$.
\end{lemma}
\begin{proof}
  Let $V'= \bigoplus_{i=1}^n V_i'$ be a finite dimensional subspace of $V$. Then for an element $v=\sum v_i$ with $v_i\in V_i'$ we have
  \[
  \Exp^*{\widetilde\ell}(v) = \sum_{j=0}^\infty \widetilde\ell\left(\frac{(v_1+\ldots+v_n)^j}{j!}\right).
  \]
  But since $\widetilde \ell$ is an $n$ quasi-homogeneous linear function we have
  \[
   \Exp^*{\widetilde\ell}(v) = \sum_{\sum ik_i = n} \widetilde\ell\left(\frac{v_1^{k_1}\ldots v_n^{k_n}}{k_1!\ldots k_n!}\right),
  \]
  and in particular $\Exp^*{\widetilde\ell}$ is a finite sum of quasi-homogeneous polynomials of degree $n$.
\end{proof}

Now we are ready to formulate the main result of this subsection.

\begin{theorem}
Let $V$ be a (possibly infinite dimensional) graded vector space. Then graded factor algebras of $\Sym(V)$ with Gorenstein duality are in bijection with quasi-homogeneous polynomials on $V$. For a quasi-homogeneous polynomial $f$, the corresponding algebra is given by
 \[
 A = \Sym(V)/\Ann(f),
 \]
Where $\Ann(f)= \{P\in \Sym(V) \,|\, \Dm_P(f)=0\}$.
\end{theorem}
\begin{proof}
  Theorem follows immediately from Theorem~\ref{thm:grgor} and Lemmas~\ref{lem:potential},~\ref{lem:qhpoly}.
\end{proof}

\begin{example}\label{ex:deg1}
  Let algebra $A=\bigoplus_{i=0}^n A_i$ be as before. Assume further that $A$ is generated in degree 1, and $V\to A_1$ is a surjective map. Then the potential $\Exp^*{\widetilde \ell}$ of $A$ on $V$ is given by
 \[
 \Exp^*{\widetilde \ell}(v) = \widetilde \ell\left(\frac{v^n}{n!}\right).
 \]
 In particular algebra $A$ can be constructed as
 \[
 A\simeq \Diff(A)/\Ann\left(\widetilde \ell\left(\frac{v^n}{n!}\right)\right).
 \]
\end{example}

The result from Example~\ref{ex:deg1} was observed in \cite{KP} where it was applied to the computation of cohomology ring of toric varieties. It was also used in \cite{KavehVolume} for the computation of the cohomology rings of full flag varieties $G/B$ and in \cite{KaveKhovanskii} for the computation of the ring of complete intersections of an algebraic variety. These results crucially use that the above rings are generated in degree 1 (after dividing grading by $2$ for cohomology rings).

\section{Potential of a toric bundle}\label{sec:toric}
In this section we will apply our explicit version of Macaulay's inverse system to the computation of the ring of even degree cohomology classes of toric bundles. In particular, we obtain complete description of  cohomology rings of toroidal horospherical varieties and the ring of conditions of horospherical homogeneous spaces.

\subsection{Virtual polytopes and cohomology ring of a toric variety}
In this subsection we briefly recall the relation of the space of virtual polytopes  to the cohomology rings of toric varieties. For more detailed background on virtual polytopes we refer to \cite{panina}. We also refer to \cite{tor-var} for the introduction to toric geometry.

Let $T\simeq(\C^*)^n$ be an algebraic torus with a character lattice $M$ and a cocharacter lattice $N:=\Hom(M,\Z)$. Let $M_\R:= M\otimes \R$ and $N_\R:=N\otimes \R$ be the corresponding vector spaces. 

Let us denote by $\Pm^+$ the set of all convex polytopes in $M_\R$. Note that $\Pm^+$ is a cone with respect the Minkowski addition and scaling by a positive factor. Moreover, $\Pm^+$ has the cancellation property, i.e.
\[
\Delta_1 +\Delta = \Delta_2 +\Delta \quad \text{if and only if} \quad \Delta_1 = \Delta_2.
\]
Thus we can consider a Grothendieck group associated to $\Pm^+$.

\begin{definition}
 The space of virtual polytopes $\Pm$ is the Grothendieck group of the semigroup of convex polytopes $\Pm^+$. 
\end{definition}

Let us fix a complete fan $\Sigma\subset N$ which is a normal fan of some polytope. Let further $\Pm_\Sigma^+$ be the set of polytopes in $M_\R$ such that $\Sigma$ refines their normal fans. It is easy to see that $\Pm_\Sigma^+$ is a subsemigroup of $\Pm_\Sigma$. Similar to the general case, the space of virtual polytopes $\Pm_\Sigma$ associated with fan $\Sigma$ is the Grothendieck group of the semigroup of convex polytopes $\Pm_\Sigma^+$.

The scaling by positive real numbers of convex polytopes $\Pm_\Sigma^+$ extends to the action of $\R$ on $\Pm_\Sigma$, making it into a vector space. In what follows we will use a more concrete description of $\Pm_\Sigma$. Using support functions, one can identify $\Pm_\Sigma^+$ with a cone of convex, piecewise linear function on $N_\R$ which are linear on all cones of $\Sigma$. This identification extends to the space of virtual polytopes, and since any piecewise linear function is a difference of two convex piecewise linear functions, we have
\[
\Pm_\Sigma = \{ h \colon N_\R \to \R \colon \text{piecewise linear function on} \; \Sigma \}.
\]
In particular, the description above allows to work with the spaces of virtual polytopes associated to fans $\Sigma$ which are not normal for any polytope.

Similarly, the space of all virtual polytopes  $\Pm$  can be identified with all functions $h \colon N_\R \to \R$ which are piecewise linear with respect to some fan $\Sigma$. For a virtual polytope $\Delta\in\Pm$, we will call the corresponding piecewise linear function $h_\Delta \colon N_\R \to \R$ \emph{the support function} of $\Delta$.

For an integral fan $\Sigma\subset N_\R$, let us denote by $Y_\Sigma$ the corresponding toric variety. We will say that $\Sigma$ is a smooth projective fan if the corresponding toric variety is such.

Each virtual polytope $\Delta\in \Pm_\Sigma$ defines a degree two cohomology class on a toric variety $Y_\Sigma$ in the following way. Let $e_1,\ldots, e_r$ be the integral ray generators of $\Sigma$ and $D_1,\ldots, D_r$ be the corresponding $T$-invariant divisors on $Y_\Sigma$. Let further $h_\Delta:N_\R\to \R$ be the support function of a (virtual) polytope $\Delta$. Then we have a map
\[
\Pm_\Sigma \to H^2(Y_\Sigma,\R), \quad \Delta \mapsto \sum h_{\Delta}(e_i) [D_i],
\]
where $[D_i]\in H^2(Y_\Sigma)$ is a class Poincar\'e dual to the homology class of the divisor $D_i$.

The following theorem is well-known (see, for example, \cite{tor-var}).

\begin{theorem}\label{thm:torcoh}
 The map $\Pm_\Sigma \to H^2(Y_\Sigma,\R)$ is surjective. Moreover, $H^*(Y_\Sigma,\R)$ is generated in degree 2, thus one has a surjective ring homomorphism:
 \[
 \Sym(\Pm_\Sigma) \to H^*(Y_\Sigma,\R).
 \]
\end{theorem}

\subsection{Preliminaries on toric bundles} In this subsection we give preliminaries on toric bundles, for more detailed treatment see Section 3 of \cite{hof2020}.

Let $T$ be an algebraic torus as before and let $E\to X$ be a $T$-principle bundle over smooth closed orientable manifold. For any manifold $Y$ with a continuous action by $T$ let $E \times_T Y \coloneqq (E \times Y) / T$ be the associated fiber bundle over $X$. It is easy to see that a fiber of the bundle $E\times_T Y \to X$ is homeomorphic to~$Y$.

For any $\lambda\in Y$, let  $\C_\lambda$ be the corresponding 1-dimensional representation of $T$ and let $\cL_\lambda :=E \times_T \C_\lambda$ be the associated complex line bundle over $X$. One can see that
\[
\cL_{\lambda+\mu}= \cL_\lambda\otimes \cL_\mu, \text{ for any } \lambda, \mu \in M.
\]
Thus we have a homomorphism $c\colon M \to H^2(X,\Z)$ of abelian groups defined via 
\[
\lambda \mapsto c_1(\cL_\lambda),
\]
where $c_1(\cL_\lambda)$ is the first Chern class of the line bundle $\cL_\lambda$. We extend the group homomorphism $c$ to a linear map
\[
c\colon M_\R \to H^2(X,\R).
\]

Let $\Sigma \subset N$ be a smooth, projective fan and let be $Y_\Sigma$ the corresponding toric variety. We denote by $E_\Sigma \to X$ the associated \emph{toric bundle}:
\[
E_\Sigma = E\times_T Y_\Sigma.
\]
The following result describes the cohomology groups of a toric bundle. It is an easy corollary of Leray-Hirsch theorem (\cite[Theorem 5.11]{BottTu}).
\begin{proposition}\label{prop:LH}
Let $T$ be an algebraic torus, $p \colon E \to X$ be a $T$-principle bundle, and $Y_\Sigma$ be a smooth projective $T$--toric variety given by a fan $\Sigma$. Then as a group the cohomology of $E_\Sigma = E \times_T Y_\Sigma$ is given by
  \begin{equation}\label{eq:lereyhirsch}
    H^*(E_\Sigma, \R) \simeq H^*(X, \R) \otimes H^*(Y_\Sigma, \R) \text{.}      
  \end{equation}
\end{proposition}

As before, to each virtual polytope $\Delta\in\Pm_\Sigma$ we can associate a degree 2 cohomology in $H^2(Y_\Sigma,\R)$. thus, using isomorphism~(\ref{eq:lereyhirsch}), to each virtual polytope $\Delta \in \Pm_\Sigma$ we can associate a degree 2 cohomology class on the corresponding toric bundle $E_\Sigma$:
\[
\rho\colon\Pm_\Sigma \to H^2(E_\Sigma),\quad \Delta \mapsto \sum 1\otimes h_{\Delta}(e_i) [D_i] \in H^*(X)\otimes H^*(Y_\Sigma)\simeq  H^*(E_\Sigma).
\]

Moreover we have the following proposition
\begin{proposition}\label{prop:cohgen}
The cohomology ring of a toric bundle $p:E_\Sigma \to X$ is generated by the image of $H^*(X,\R) \oplus \Pm_\Sigma$ under 
\[
p^*\oplus \rho\colon H^*(X,\R) \oplus \Pm_\Sigma \to H^*(E_\Sigma,\R).
\]
\end{proposition}
\begin{proof}
The proposition follows immediately from Theorem~\ref{thm:torcoh} and Proposition~\ref{prop:LH}.
\end{proof}

\subsection{Potential on toric bundles}
In this section we assume that $\dim_\R X = 2k$ is even. In that case, $\dim_\R E_\Sigma = 2k+2n$ is also even. 

For a manifold $M$ of even dimension, let us denote by $A^*(Y)$ the subring of even degree cohomology classes with degrees divided by 2:
\[
A^*(Y) := \bigoplus_{i=0}^{(\dim_\R Y)/2} H^{2i}(Y,\R).
\]
The ring $A^*(Y)$ is a commutative ring with a Gorenstein duality given in the following way. For a manifold $Y$, we will denote by $\ell_Y : A^*(Y)\to \R$ the linear functional given by the pairing with the fundamental class of $Y$. That is
\[
\ell_Y(\gamma) = \int_{Y} \gamma, \quad \text{ for } \gamma \in A^*(Y).
\]
By Poincar\'e duality, the function $\ell_Y$ provides a Gorenstein duality on  $A^*(Y)$ when $\dim_\R(Y )$ is even. 

Now let $E\to X$ be a $T$-principle bundle over smooth closed orientable manifold $X$ of even dimension $2k$. Let further $\Sigma\subset N$ be a smooth projective fan, and let $p:E_\Sigma \to X $ be the corresponding toric bundle. Similar to Proposition~\ref{prop:cohgen}, the ring of even cohomologies $A^*(E_\Sigma)$ is generated by the image of $A^*(X) \oplus \Pm_\Sigma$ under $p^*\oplus \rho$. Thus we  have the surjection:
\[
\Sym(A^*(X)\oplus \Pm_\Sigma) \to A^*(E_\Sigma).
\]
We will describe the kernel of the above surjection using the explicit computation of Macaulay's inverse system given in Theorem~\ref{thm:potential} and a version of the BKK theorem for toric bundles obtained in \cite{hof2020}.

First we will need some definitions. For a virtual polytope $\Delta \in \Pm_\Sigma$, define its $i$-th horizontal part $h_i(\Delta)$ to be the unique class in $A^{*}(X)$ such that
\[
  \int_{E_\Sigma} \rho(\Delta)^{n+i}\cdot p^*(\eta) = \int_X h_i(\Delta) \cdot \eta \text{,}
\]
for any $\eta\in A^{*}(X,\R)$. It is clear that $h_i(\Delta)$ is a homogeneous element of degree $i-n$ in $A^*(X)$. In particular, $h_i(\Delta) = 0$ for $i<n$. The horizontal parts of $\Delta$ were computed in \cite{hof2020}.

\begin{theorem}[\cite{hof2020}]\label{BKKhor}
  For any $\Delta \in \Pm_\Sigma$, its $n+i$-th horizontal part can be computed as
  \[
    h_{n+i}(\Delta) = \frac{(n+i)!}{i!} \int_\Delta c(\lambda)^i \diff \lambda \text{,}
  \]
  where, as before, $c:M_\R\to A^1(X)$ is the linear map given by $c(\lambda)=c_1(\cL_\lambda)$ for integral $\lambda$.
\end{theorem}

Denote by $P_X$ the potential of $A^*(X)$ with respect to the generating vector space $A^*(X)$ and by $P_{E_\Sigma}$ the potential of $A^*(E_\Sigma)$ with respect to $A^*(X) \oplus \Pm_\Sigma$. The following theorem is the main result of this section. It provides a relation between potentials $P_{E_\Sigma}$ and $P_X$.

\begin{theorem}\label{thm:torpot}
  For any $\Delta \in \Pm_\Sigma$ and $\gamma\in A^*(X)$ we have
  \[
P_{E_\Sigma}(\gamma,\Delta) = \int_\Delta  P_X(c(\lambda)+\gamma) \diff \lambda.
  \]
\end{theorem}

\begin{proof}
The proof relies on Theorem~\ref{BKKhor} and is a direct calculation otherwise. Recall that by Theorem~\ref{thm:potential}, the potentials $P_{E_\Sigma}, P_X$ are given by
\[
P_{E_\Sigma}= \Exp^*\ell_{E_\Sigma},\quad P_X= \Exp^*\ell_X
\]
In other words, the evaluations of potentials $P_{E_\Sigma}$ and $P_X$ on an elements of $A^*(X) \oplus \Pm_\Sigma$ and $A^*(X)$ respectively are given by the following formulas 
\[
P_{E_\Sigma}(\gamma,\Delta) = \ell_{E_\Sigma} \left( \sum_{j=0}^{\infty} \frac{(p^*(\gamma)+\rho(\Delta))^{j}}{j!}\right), \quad P_X(\gamma) = \ell_X \left( \sum_{j=0}^{\infty} \frac{\gamma^{j}}{j!}\right)
\]
By definition of horizontal parts $h_i(\Delta)$ of $\Delta$ we have
  \[
P_{E_\Sigma}(\gamma,\Delta) =  \ell_{E_\Sigma} \left( p^*(\exp(\gamma)) \cdot \sum_{j=0}^{\infty} \frac{\rho(\Delta)^{j}}{j!}\right) = \ell_X \left(\exp(\gamma) \cdot \sum_{j=0}^{\infty} \frac{h_j(\Delta)}{j!}\right).
  \]
Now, by Theorem~\ref{BKKhor} we get $h_j(\Delta) = 0$ for $j<n$ and
 \[
    h_{n+i}(\Delta) = \frac{(n+i)!}{i!} \int_\Delta c(\lambda)^i \diff \lambda \text{,}
 \]
otherwise. Hence we obtain
\begin{equation*}
\begin{split}
P_{E_\Sigma}(\gamma,\Delta) = \ell_X \left(\exp(\gamma) \cdot \sum_{j=0}^{\infty} \frac{h_j(\Delta)}{j!}\right) &= \ell_X \left(\exp(\gamma) \cdot \sum_{i=0}^{\infty} \frac{h_{n+i}(\Delta)}{(n+i)!}\right) = \\
\ell_X \left(\exp(\gamma) \cdot \sum_{i=0}^{\infty} \frac{1}{i!} \int_\Delta c(\lambda)^i \diff \lambda \right) &= \ell_X \left(\int_\Delta\exp(\gamma+c(\lambda)) \diff \lambda\right) = \\
\int_\Delta  \ell_X(\exp(c(\lambda)+\gamma)) \diff \lambda &= \int_\Delta  P_X(c(\lambda)+\gamma) \diff \lambda,
\end{split}    
\end{equation*}
which finishes the proof.
\end{proof}

\begin{remark}
Theorem~\ref{thm:torpot} suggests that the function $P_{E_\Sigma}\colon A^*(X) \oplus \Pm_\Sigma \to \R$ defined by
  \[
P_{E_\Sigma}(\gamma,\Delta) = \int_\Delta  P_X(c(\lambda)+\gamma) \diff \lambda.
  \]
is a polynomial. Indeed, it follows from the results of \cite{PK92} that the integral of a polynomial over a polytope defines a polynomial function on the space of virtual polytopes.
\end{remark}

As an immediate corollary of Theorems~\ref{thm:torpot} and~\ref{thm:gor} we obtain a computation of the ring of even degree cohomologies of toric bundle. Alternative descriptions were obtained in \cite{US03,roch,hof2020}.

\begin{theorem}\label{thm:cohtorbun}
 Let $E\to X$ be a principle $T$-bundle with $\dim_\R M$ even and let $E_\Sigma$ be a toric bundle given by a smooth projective fan $\Sigma$. Then the ring of even degree cohomology classes $A^*(E_\Sigma)$ is given by
 \[
 A^*(E_\Sigma) = \Sym\big(A^*(X) \oplus \Pm_\Sigma\big)/ \Ann(P_{E_\Sigma}),
 \]
 where the polynomial $P_{E_\Sigma}$ is given by
 \[
P_{E_\Sigma}(\gamma,\Delta) = \int_\Delta  P_X(c(\lambda)+\gamma) \diff \lambda.
 \]
\end{theorem}

\subsection{The ring of conditions of horospherical homogeneous spaces}
The ring of conditions is a version intersection theory for not-necessarily complete homogeneous space. It was introduced by De Concini and Procesi in \cite{DP85}  as a solution to Hilbert's 15'th problem, i.e. as a ring in which one can study classical enumerative problems. As another corollary of our main theorem, we obtain a description of the ring of conditions of horospherical homogeneous space.

Let $G$ be a reductive group over $\C$ and let $U\subset G$ be its maximal unipotent subgroup. The subgroup $H\subset G$ is called {\emph horospherical} if it contains $U$. The homogeneous space $G/H$ is called horospherical if $H$ is a horospherical subgroup of $G$. Let $P_H$ be the normalizer of $H$, it is well-known that $P_H$ is a parabolic subgroup of $G$, and that $T_H=P_H/H$ is an algebraic torus. Moreover, the homogeneous space $G/H$ is a total space of $T_H$-principal bundle over generalized flag variety $G/P_U$ (see for example \cite{Pasq,pasquier} for details).

Toric bundles $E_\Sigma:=G/H\times_{T_H} Y_\Sigma$ associated to the principle $T_H$-bundle $G/H\to G/P_H$ (also called toroidal horospherical varieties) are known to have only even degree cohomology classes. Thus Theorem~\ref{thm:cohtorbun} computes full cohomology rings in this case. In fact, by Bia {\l}ynicki-Birula decomposition (\cite{BB73}),  varieties $E_\Sigma$ have a paving by affine spaces, so that $H^*(E_\Sigma,\Z)$ is torsion free. Hence, Theorem~\ref{thm:cohtorbun}, in particular, computes integer cohomology of $E_\Sigma$. Moreover, in the situation as above the ring of conditions $C^*(G/H)$ is given by
\[
C^*(G/H) = \varinjlim A^*(E_\Sigma),
\]
where the limit is taken over all smooth projective fans $\Sigma$ (see \cite[Section 6.3]{DP85}).

Thus we obtain the following theorem.
\begin{theorem}\label{thm:cond}
Let $G/H$ be a horospherical homogeneous space and let $E_\Sigma$ be its smooth projective toroidal compactification. Then the cohomology ring $H^{*/2}(E_\Sigma,\R)=A^*(E_\Sigma)$ of $E_\Sigma$ is given by
 \[
A^*(E_\Sigma) = \Sym\big(A^*(G/P_H) \oplus \Pm_\Sigma\big)/ \Ann(P_{G/H}),
\]
where $A^*(G/P_H)$ is the cohomology ring of $G/P_H$ and
 \[
P_{G/H}(\gamma,\Delta) = \int_\Delta  P_{G/P_H}(c(\lambda)+\gamma) \diff \lambda.
 \]
Moreover the ring of conditions of $G/H$ is given by 
  \[
C^*(G/H) = \Sym\big(A^*(G/P_H) \oplus \Pm\big)/ \Ann(P_{G/H}).
\]
\end{theorem}

\bibliographystyle{alpha}
\bibliography{tb}

\Addresses

\end{document}